\documentclass[11pt,makeidx]{amsart}
\usepackage{amsthm,amsfonts,amssymb,amsmath,amscd,amsrefs,thmtools}
\usepackage[all]{xy}
\usepackage{xr-hyper}
\usepackage{color}
\definecolor{crimsonglory}{rgb}{0.75, 0.0, 0.2}
\definecolor{darkblue}{rgb}{0.0, 0.0, 0.55}
	\definecolor{deepskyblue}{rgb}{0.0, 0.75, 1.0}

\usepackage[colorlinks=true,linkcolor = blue, citecolor = deepskyblue]{hyperref}

\usepackage[OT2,OT1]{fontenc}
\usepackage{mathrsfs}

\newtheoremstyle{mystyle}
  {8pt}% measure of space to leave above the theorem. E.g.: 3pt
  {3pt}% measure of space to leave below the theorem. E.g.: 3pt
  {\em}% name of font to use in the body of the theorem
  {}% measure of space to indent
  {\scshape}% name of head font
  {.}% punctuation between head and body
  {3pt}% space after theorem head; " " = normal interword space
  {}% Manually specify head
  
\theoremstyle{mystyle}

\newcounter{exe}
  
\newtheorem{thm}{\textsc{Theorem}}

\newtheorem{cor}[exe]{\textsc{Corollary}}
\newtheorem{lemma}[exe]{\textsc{Lemma}}
\newtheorem{prop}[exe]{\textsc{Proposition}}

\theoremstyle{definition}

\usepackage{color}

\title[]{On Artin's Conjecture for Pairs of Diagonal Forms}
\author{Jo\~ao Campos Vargas}\email{jcvargas@princeton.edu}
\address{Department of Mathematics, Princeton University, Princeton, NJ 08540.}
\date{\today}

\newcommand{\ord}{\mathop{\mathrm{ord}}}

\begin{document}

\begin{abstract}

%Mention the other result here

Let $p$ be an odd prime and $d = p^{\tau}(p-1)$. In the spirit of Aritn's conjecture, consider the system of two diagonal forms of degree $d$ in $s$ variables given by \begin{equation*}\begin{split}
a_1x_1^d + \cdots + a_sx_s^d = 0\\
b_1x_1^d + \cdots + b_sx_s^d = 0
\end{split}
\end{equation*} with $a_i, b_i \in \mathbb{Q}_p$.
For $s > 2 \frac{p}{p-1}d^2 - 2d$, this paper shows that this system has a non-trivial $p$-adic solution for every $\tau \ge 3, p \ge 7$, and for every $\tau = 2, p \ge \frac{C}{2}+4$, where $C \le 9996$. Moreover, for $s > (2\frac{p}{p-1} + \frac{C-3}{2p-2})d^2 - 2d$, this system will have a non-trivial $p$-adic solution for every $\tau = 1, p \ge 5$.\end{abstract}

\maketitle

\section{Introduction}\label{intro}

\par The integer solvability of a homogeneous system of diagonal forms is a variation of the Waring's problem that is interesting in its own right. Even though the Hasse principle may not apply directly, a natural first step is to find conditions for which such systems always have non-trivial $p$-adic solutions for all primes $p$. E. Artin conjectured that a system of $R$ homogeneous diagonal forms of degree $d$ has a non-trivial $p$-adic solution for every prime $p$ as long as the number of variables is at least $s \ge Rd^2 + 1$. However, this is only known to be true when $R = 1$. The best results towards this conjecture were given by Knapp (c.f. Theorem 2 in \cite{knapp2001systems}) and Skinner (c.f. Theorem A in \cite{skinner2020solvability}), proving that lower bounds for $s$ close to $R^2d^2$ are sufficient.

\par This paper explores this conjecture for the case $R = 2$. Consider a system of diagonal forms of degree $d$ given by\begin{equation}
\begin{split}
a_1x_1^d + \cdots + a_sx_s^d = 0\\
b_1x_1^d + \cdots + b_sx_s^d = 0
\end{split} \tag{$1$}
\end{equation}
with $a_i, b_i \in \mathbb{Q}_p$. We look for conditions on $s$ for which such system always has a non-trivial $p$-adic solution. Davenport and Lewis started this exploration in \cite{davenport1967two}, proving that Artin's conjecture holds whenever $d$ is odd. More recently, Br\"{u}dern and Godinho showed in \cite{brudern2002artin} that the conjecture is true for all odd primes $p$ and $d \ne p^{\tau}(p-1), \tau \ge 1$. For $d = p^{\tau}(p-1), \tau \ge 1$ the results in \cite{brudern2002artin} prove that the bound $s \ge 4d^2$ is sufficient to guarantee a non-trivial $p$-adic solution for every odd prime $p$. Motivated by this last result, this paper works on closing the gap between $4d^2$ and $2d^2 + 1$ when $d = p^{\tau}(p-1), \tau \ge 1$.

\par This paper was inspired by a weaker result which appears in the work of Godinho and de Souza Neto in \cite{godinho2013pairs} and shows that the bound $s > 2\frac{p}{p-1}d^2 - 2d$ is sufficient for large values of $\tau$ (roughly $\tau > \frac{p}{2}$). The main ingredient of this paper came in the form of \mbox{Proposition \ref{main-lemma}}, showing that the same lower bound for $s$ is sufficient for all $\tau \ge 3$. For $\tau = 1, 2$ a new tool was derived from a combinatorial result of Alon and Dubiner (c.f. Theorem 1 in \cite{alon1995lattice}). In application this takes the form of Lemma \ref{alon1}, and it allows us to deal with large primes efficiently.

\par The main result of this paper is described in the next theorem.

\begin{thm} \label{main} Let $p$ be an odd prime. Consider the system $(1)$ with $d = p^{\tau}(p-1), \tau \ge 1$. Let $C \ge 3$ be a constant for which Lemma \ref{alon1} holds. Assume that $(1)$ satisfies one of the following conditions:
\begin{enumerate}
 \item[(i)] $\tau \ge 3, s > 2\frac{p}{p-1}d^2 - 2d$, and $p \ge 7$.
\item[(ii)] $\tau = 2, s > 2\frac{p}{p-1}d^2 - 2d$, and $p \ge \frac{C}{2} + 4$.
\item[(iii)] $\tau = 1, s > (2\frac{p}{p-1} + \frac{C-3}{2p-2})d^2 - 2d$, and $p \ge 5$.
\end{enumerate}
Then the system $(1)$ has a non-trivial $p$-adic solution.
\end{thm}

\par It will be proved later in Lemma \ref{alon1} that $C \le 9996$. Combining the results established in \cite{brudern2002artin} with Theorem \ref{main} has an interesting consequence towards Artin's conjecture:

\begin{thm} \label{artin-eps} Let $\epsilon > 0$. Assume that the system $(1)$ has at least $s > (2+\epsilon)d^2$ variables. Then $(1)$ has a non-trivial $p$-adic solution for all sufficiently large primes $p$.
\end{thm}

\begin{proof} Let $p \ge \frac{C}{2} + 4$ be such that $\frac{C+1}{2p-2} < \epsilon$. By Theorem \ref{main} the system $(1)$ has a non-trivial solution for all $d = p^{\tau}(p-1), \tau \ge 1$. For other values of $d$, Theorem 1 in \cite{brudern2002artin} shows that the condition $s > 2d^2$ suffices to ensure a non-trivial solution. The result follows.
\end{proof}

\par At this point it is worth remarking a possible improvement for the theory via Lemma \ref{alon1}. In fact, Lemma \ref{olson} indicates that $C = 3$ might be sufficient for Lemma \ref{alon1}. If that is the case, then Theorem \ref{main} proves that the bound $s > 2\frac{p}{p-1}d^2 - 2d$ suffices for all $d = p^{\tau}(p-1), \tau \ge 1$, and $p \ge 7$. The work of Godinho and de Souza Neto (c.f. Theorem 1.1 in \cite{godinho2011pairs}) already shows that this lower bound for $s$ is also sufficient for all values of $d$ for $p = 3, 5$. Hence we would have that for every $p$ odd, the bound $s > \max \{ 2\frac{p}{p-1}d^2 - 2d, 2d^2\}$ is sufficient to guarantee a non-trivial solution for the system $(1)$.

\section{Addendum}

\par Shortly after the submission of this work, a paper by M. S. Kaesberg (c.f. \cite{kaesberg2020artin}) appeared on arXiv claiming to prove the sufficiency of the bound $s \ge 2d^2+1$ for all odd primes. Even though both papers seem to follow the same outline layed out Br\"{u}udern and Godinho, the methods used by Kaesberg are different from the ones presented here. The main difference comes with Kaesberg's refinement of classes of vectors, introducing the notion of color nuances. This allows greater control over specific systems of congruences, and many technical lemmas on Kaesberg's work rely on this classification. On the other hand, the foundation of this paper rests on Proposition \ref{main-lemma}, which has a natural combinatorial meaning and, differently from most other results presented here, an algebraic proof. Although this result alone does not imply Artin's conjecture -- and in fact does not appear in Kaesberg's work -- it is the author's hope that it can also be useful in similar problems.

\section{Normalization}\label{normalization}

The solvability of a system $(1)$ is equivalent to the solvability of many others. Identify $(1)$ with \[A = \begin{pmatrix} a_1 & a_2 & \cdots & a_s \\ b_1 & b_2 & \cdots & b_s \end{pmatrix} \in M_{2 \times s} (\mathbb{Q}_p),\] its corresponding matrix. Let $M \in \text{GL}_2(\mathbb{Q}_p)$ represent a non-singular linear combination of the rows of $A$. Let $P$ be a permutation matrix and $D = \text{diag}(p^{d\nu_1}, \cdots, p^{d\nu_s})$ with $\nu_i \in \mathbb{Z}$. Notice that the system $A$ has a non-trivial $p$-adic solution if and only if $MADP$ does. This defines an equivalence relation in $M_{2\times s}(\mathbb{Q}_p)$.

\par Let $A \in M_{2 \times s}(\mathbb{Q}_p)$. For $I \subset \{1, \cdots, s\}$ define $A_I$ to be the matrix obtained by choosing the columns of $A$ labeled by $I$. Define the function
\[ \theta(A) = \sum_{I \subset \{1, \cdots, s\} \atop \#I = 2} \text{ord}_p(\det(A_I)) \in \mathbb{Z} \cup \{\infty \}.\]

The function $\theta$ behaves nicely within equivalence classes. In fact, for $M \in \text{GL}_2(\mathbb{Q}_p)$, $P \in \text{GL}_s(\mathbb{Z})$, $D = \text{diag}(p^{d\nu_1}, \cdots, p^{d\nu_s})$ as above, $A' = MADP$, we have
\[ \theta(A') = \theta(A) + \binom{s}{2} \text{ord}_p(\det(M)) + (s-1)(\nu_1 + \cdots + \nu_s).\]

\par Observe that $\theta(A) = \infty$ hardly happens. In fact, due to the compactness of $\mathbb{Z}_p$, the non-trivial solvability of all systems in $M_{2\times s} (\mathbb{Q}_p)$ is implied by the non-trivial solvability of those with $\theta(A) < \infty$ through a standard limiting argument. For a complete proof, c.f. Section 5 in \cite{davenport1967two}.

\par A system $(1)$ is said to be normalized if its corresponding matrix has entries in $\mathbb{Z}_p$ and the least, finite $\theta$ value among all equivalent systems with entries in $\mathbb{Z}_p$. Normalized systems have a few desirable properties and from now on we assume $(1)$ to be normalized.

\section{p-adic solvability} \label{solvability}

In order to solve $(1)$ $p$-adically it suffices to find a solution mod $p^{\tau + 1}$ with non-singular support in $\mathbb{F}_p$. A variation of Hensel's lemma provides us with this result (c.f. Lemma 7 in \cite{davenport1967two}):

\begin{lemma}\label{reduction} Let $\ord_p(d) = \tau$. Consider the reduction of $(1)$ mod $p^{\tau + 1}$: 
\begin{equation}
\begin{split}
a_1x_1^d + \cdots + a_sx_s^d \equiv 0 \pmod{p^{\tau + 1}}\text{ }\\
b_1x_1^d + \cdots + b_sx_s^d \equiv 0 \pmod{p^{\tau + 1}}.
\end{split} \tag{$2$}
\end{equation}
Let $\mathbf{x} = (x_1, \cdots, x_s)$ be a solution for $(2)$ and let $I = \{ i  :  x_i \not \equiv 0 \pmod{p}\}$. Let $A_I \in M_{2 \times |I|}(\mathbb{Z}_p)$ be the matrix formed by the columns of $A$ with indices in $I$. Assume that $A_I$ has rank two when reduced mod $p$. Then $(1)$ has a non-trivial $p$-adic solution.
\end{lemma}

For this reason, variables for which at least one of the corresponding coefficients is not divisible by $p$ play an important role throughout the argument. These will be referred to as variables at level zero, as defined in the next section. The proof of Theorem ~\ref{main} focuses on finding a solution for $(2)$ satisfying the conditions of Lemma \ref{reduction}.

\section{Classification of Variables} \label{variables}

A variable $x_i$ is said to be at level $l$ whenever \[\min\{ \text{ord}_p(a_i), \text{ord}_p(b_i)\} = l.\] Let $\mathcal{T}_l$ be the set of variables at level $l$ and let $m_l = \#\mathcal{T}_l$. For a subset $\mathcal{H} \subset \mathcal{T}_0$ define $q(\mathcal{H})$ to be the least possible number of elements not divisible by $p$ in the array $\{\lambda a_i + \mu b_i\}_{i \in \mathcal{H}}$, where $\lambda$ and $\mu$ range over the integers and are not both divisible by $p$. Define $q_0 = q(\mathcal{T}_0)$. The normalization of the system $(1)$ gives us (c.f. Lemma 9 in \cite{davenport1967two}) \[m_0 + m_1 + \cdots + m_l \ge \frac{(l+1)s}{d} \text{  and  } q_0 \ge \frac{s}{2d}.\]
The system ($2$) can be rewritten by levels as
\begin{equation*}
\begin{split}
\sum_{l = 0}^{\tau} p^l \sum_{i \in \mathcal{T}_l} c_ix_i^d \equiv 0 \pmod{p^{\tau + 1}}\text{ }\\
\sum_{l = 0}^{\tau} p^l \sum_{i \in \mathcal{T}_l} d_ix_i^d \equiv 0 \pmod{p^{\tau + 1}}.
\end{split}
\end{equation*}

\par The main idea is to use the different levels to solve the system. We provide a rough description of this process. First, we start at level zero and solve a system of congruences mod $p$, looking for a solution with a small number of non-zero variables. These non-zero variables will be multiplied by a new variable which will provide us with a new variable at a certain level $l \ge 1$. We repeatedly apply this process at level zero, and after exploiting all available variables at that level we move on to the next levels.

\par The process of using the non-zero variables of a solution to create a new variable at a higher level is referred to as a contraction. We remark that there will be different goals for a contraction. At level zero, the main goal of a contraction is to create new variables at higher levels from a non-singular solution. These new variables will be referred to as primary variables. The new variables generated after any further contractions at higher levels using at least one primary variable will also bear the name of primary variables.

\par In summary as long as a variable can be traced to a non-singular solution at level zero, this variable will be called primary. The remaining variables are called secondary and will also be contracted among themselves. The main goal is to create a primary variable at level at least $\tau + 1$. In that case, the system $(2)$ can be solved non-singularly, as such variable traces back to variables forming a non-singular solution at level zero. Moreover, since primary variables always trace back to level zero, this level has an important role in the argument. Contractions in the remaining levels follow a slightly different pattern in which secondary variables will be important both to ensure a contraction using a primary variable and to contract among themselves.

\par The next section contains some preparation for Sections \ref{main-section} and \ref{zero-sums}. The main ingredient of this paper, which allows us to deal with small values of $\tau$, is presented separetely in Section~\ref{main-section}. Different from the other results used for contracting variables, its proof is fundamentally algebraic. Section \ref{zero-sums} contains other combinatorial results needed for contractions, including Lemma \ref{alon1} which allows us to deal with the cases $\tau = 1, 2$.

\section{Some Classical Results} \label{classical}

\par This section is dedicated to stating a few classical results that will be used throughout the argument. The first one is a theorem of Cauchy rediscovered by Davenport (c.f. Theorem A in \cite{davenport1935addition}).

\begin{lemma}[The Cauchy-Davenport Theorem] \label{C-D}Let $A, B \subset \mathbb{F}_p$ and let $A + B = \{a + b : a \in A, b \in B\}$. Then \[|A+B| \ge \min\{p, |A| + |B| - 1\}.\]
\end{lemma}

\par We now derive a few consequences from this result.

\begin{lemma}\label{C-D-1} Let $a_1, \cdots, a_{p-1}$ be integers not divisible by a prime $p$ and let $c$ be an integer. Then there exists a solution for the equation \[a_1x_1^{p-1} + \cdots + a_{p-1}x^{p-1} \equiv c \pmod{p}.\]
\end{lemma}

\begin{proof} Observe that $x^{p-1} \in \{0, 1\} \pmod{p}$. Let $A_i = \{0, a_i\} \pmod{p}$. Then, by repeatedly applying Lemma \ref{C-D}, we obtain \[ |A_1 + \cdots + A_{p-1}| = p.\] The result follows.
\end{proof}

\begin{cor}\label{C-D-2} Let $a_1, \cdots, a_p$ be integers. Then the equation \[ a_1x_1^{p-1} + \cdots + a_px_p^{p-1} \equiv 0 \pmod{p} \] has a non-trivial solution.
\end{cor}

\par A few definitions are convenient at this point. Let $v \in \mathbb{Z}_p^r$.  The vector $v$ is said to be primitive if $p$ does not divide all its entries, i.e. if its reduction in $\mathbb{F}_p^r$ is non-zero. In that case the projective class of $v$ consists of all reductions $\lambda v$ in $\mathbb{F}_p^r$ with $\lambda$ not divisible by $p$. We then have the following result:

\begin{cor} \label{C-D-3} Let $v_i = \binom{a_i}{b_i} \in \mathbb{F}_p^2 - \{\mathbf{0}\}, i = 1, \cdots, 2p-1$. Assume that the vectors $v_{p+1}, \cdots, v_{2p-1}$ lie in the same projective class. Then the system 
\begin{equation*}
\begin{split}
a_1x_1^{p-1} + \cdots + a_{2p-1}x_{2p-1}^{p-1} \equiv 0 \pmod p \\
b_1x_1^{p-1} + \cdots + b_{2p-1}x_{2p-1}^{p-1} \equiv 0 \pmod p
\end{split}
\end{equation*} has a solution with $(x_1, \cdots, x_p) \not \equiv (0, \cdots, 0) \pmod p$.
\end{cor}

\begin{proof} Through a change of basis we can assume without loss of generality that the vectors $v_{p+1}, \cdots, v_{2p-1}$ are in the projective class $[1:0]$. In other words, $b_{p+1} = \cdots = b_{2p-1} = 0$. By the previous corollary, the equation \[b_1x_1^{p-1} + \cdots + b_{p}x_{p}^{p-1} \equiv 0 \pmod{p}\] has a non-trivial solution $(x_1, \cdots, x_p)$. Set $c = - (a_1x_1^{p-1} + \cdots + a_{p}x_{p}^{p-1})$. By Lemma \ref{C-D-1}, it follows that \[a_{p+1}x_{p+1}^{p-1} + \cdots + a_{2p-1}x_{2p-1}^{p-1} \equiv c \pmod{p}\] has a solution $(x_{p+1}, \cdots, x_{2p-1})$. Then $(x_1, \cdots, x_{2p-1})$ is a solution for the system.
\end{proof}

Lastly, we state an unrelated result concerning polynomials over $\mathbb{F}_p$:

\begin{lemma}[Davenport's Principle]\label{Davenports-Principle} Let $F \in \mathbb{F}_p[x_1, \cdots, x_n]$. Assume that $\deg_{x_k} F \le p-1$ for all $1 \le k \le n$. Moreover, assume that $F(\mathbf{a}) = 0$ for all $\mathbf{a} \in \mathbb{F}_p^n$. Then $F \equiv 0 \in \mathbb{F}_p[x_1, \cdots, x_n]$.
\end{lemma}

\begin{proof}The proof goes by induction on $n$. For $n = 1$ the result is immediate as a non-zero polynomial $F$ of degree $\deg F \le p-1$ has at most $p-1$ roots. For the general case, write \[F = x_n^{p-1}F_{p-1} + x_n^{p-2} F_{p-2}+\cdots + F_0,\]
where $F_i \in \mathbb{F}_p[x_1, \cdots, x_{n-1}]$. Notice that the equation above, seen as a polynomial in $x_n$, has $p$ roots but degree $p-1$. Therefore $F_i(\mathbf{a}_n) = 0$ for all $\mathbf{a}_n \in \mathbb{F}_p^{n-1}$. By induction, $F_i \equiv 0$ and so $F \equiv 0 \in \mathbb{F}_p[\mathbf{x}]$.
\end{proof}

\section{Two Equations Mod p} \label{main-section}

In order to solve systems of congruences mod $p$ efficiently, we prove the following result:

\begin{prop}\label{main-lemma} Consider the system
\begin{equation*}
\begin{split}
a_1x_1^{p-1} + \cdots + a_{p}x_{p}^{p-1} + a_{p+1} x_{p+1}^{p-1} + \cdots + a_{3p-3}x_{3p-3}^{p-1} \equiv 0 \pmod{p}\\
b_1x_1^{p-1} + \cdots + b_{p}x_{p}^{p-1} + b_{p+1} x_{p+1}^{p-1} + \cdots + b_{3p-3}x_{3p-3}^{p-1} \equiv 0 \pmod{p}
\end{split} \tag{$3$}
\end{equation*}
where $a_i$ and $b_i$ are not both divisible by $p$ for any $i$. Then $(3)$ has a solution with $(x_1, \cdots, x_p) \not \equiv (0, \cdots, 0) \pmod{p}$.
\end{prop}

\par In application, the first $p$ variables in this proposition will be primary variables and the remaining ones will be secondary variables. This means that a solution with at least one of the $p$ first variables being non-zero will yield a primary variable at a higher level. Proposition \ref{main-lemma} shows that in order to make a contraction creating a primary variable it is sufficient to have $p$ primary variables and $2p-3$ secondary variables at the same level. This result will be especially useful at level $\tau$, but it will also make it easier to contract as many primary variables as possible at level $\tau - 1$.

\par In order to prove Proposition \ref{main-lemma}, we will need an auxiliarly result:

\begin{lemma}\label{tech-lemma} Let $v_i = \binom{v_i(1)}{v_i(2)} \in \mathbb{F}_p^2 - \{ \mathbf{0} \}, i = 1, \cdots, 3n$ and $v_1(2) \cdots v_n(2) \ne 0$. Assume that no $n+1$ of these vectors lie in the same projective class. Then, there exist indices $\{i_1, \cdots, i_n\} \subset \{n+1, \cdots, 3n\}$ such that \[\sum_{I \in \{1, 2\}^{2n} \atop \#\{i : I(i) = 1\} = n} \prod_{i = 1}^n v_i(I(i)) \prod_{j=1}^n v_{i_j}(I(n+j)) \ne 0.\]
\end{lemma}

\begin{proof} We show by induction that for every $0 \le k \le n$ there exists a set $\{i_1, \cdots, i_k\} \subset \{n+1, \cdots, 3n\}$ such that \[\sum_{I \subset \{1, 2\}^{n+k} \atop \#\{i : I(i) = 1\} = k} \prod_{i = 1}^n v_i(I(i)) \prod_{j=1}^k v_{i_j}(I(n+j)) \ne 0.\]
\par For $k = 0$ this is obvious since the expression above equals $v_1(2) \cdots v_n(2) \ne 0$. Now assume that the hypothesis is true for $k \le n-1$. Then \[\sum_{I \in \{1, 2\}^{n+k+1} \atop \#\{i : I(i) = 1\} = k+1} \prod_{i = 1}^n v_i(I(i)) \prod_{j=1}^{k+1} v_{i_j}(I(n+j)) = v_{i_{k+1}}(1) c_1 + v_{i_{k+1}}(2) c_2, \ \text{where }\]
\[c_1 = \sum_{I \in \{1, 2\}^{n+k} \atop \#\{i : I(i) = 1\} = k} \prod_{i = 1}^n v_i(I(i)) \prod_{j=1}^k v_{i_j}(I(n+j)) \text{ and }\] \[c_2 = \sum_{I \in \{1, 2\}^{n+k} \atop \#\{i : I(i) = 1\} = k+1} \prod_{i = 1}^n v_i(I(i)) \prod_{j=1}^k v_{i_j}(I(n+j)).\]

\par By induction we can choose $\{i_1, \cdots, i_k\} \subset \{n+1, \cdots, 3n\}$ such that $c_1 \ne 0$. Therefore $v_{i_{k+1}}(1) c_1 + v_{i_{k+1}}(2) c_2 = 0$ if and only if $v_{i_{k+1}}$ lies in a specific projective class. However there are $2n-k \ge n+1$ possible choices for $v_{i_{k+1}}$ and so not all of them lie in the same projective class. Hence there is a choice of $v_{i_{k+1}}$ such that the expression above is not zero, as desired. 
\end{proof}

We are now ready to prove Proposition \ref{main-lemma}.

\begin{proof}[Proof of Proposition \ref{main-lemma}] Let $v_i = \binom{a_i}{b_i} \in \mathbb{F}_p^2 - \{\mathbf{0}\}, i = 1, \cdots, 3p-3$. Suppose first that $p$ of these vectors, say $v_{i_1}, \cdots, v_{i_p}$, lie in the same projective class. Through a change of basis it can be assumed that these elements lie in the class $[1:0]$. In the case $i_j \le p$ for some $j$ it is enough to solve the system
\[a_{i_1}x_{i_1}^{p-1} + \cdots + a_{i_p}x_{i_p}^{p-1} \equiv 0 \pmod{p}\]
with $x_{i_j} \not \equiv 0 \pmod{p}$. Such a solution follows directly from Lemma \ref{C-D-1}. On the other hand, if $i_j \ge p + 1$ for all $j$, Corollary \ref{C-D-3} guarantees a solution satisfying $(x_1, \cdots, x_p) \not \equiv (0, \cdots, 0) \pmod{p}$. Thus it can be assumed that no $p$ vectors $v_i$ lie in the same projective class.

\par Let $\mathbf{x} = (x_1, \cdots, x_{3p-3})$. Define the polynomials 
\[f(\mathbf{x}) = a_1x_1^{p-1} + \cdots + a_{3p-3}x_{3p-3}^{p-1}, \ g(\mathbf{x}) = b_1x_1^{p-1} + \cdots + b_{3p-3}x_{3p-3}^{p-1} \in \mathbb{F}_p[\mathbf{x}].\]
Let $F(\mathbf{x}) = (1 - f(\mathbf{x})^{p-1})(1-g(\mathbf{x})^{p-1})$ be the indicator function of a solution for the system $(3)$. Consider the reduction of monomials given by
\[x_1^{r_1}\cdots x_{3p-3}^{r_{3p-3}} \mapsto x_1^{t(r_1)} \cdots x_{3p-3}^{t(r_{3p-3})}, \text{ where}\]
\[t(r) = \begin{cases} r \pmod{p-1} \in \{1, \cdots, p-2\} \text{ if } p-1 \nmid r, \\ p-1 \text{ if } p-1 \mid r \text{ and } r \ge p-1, \\ 0 \text{ otherwise.}\end{cases}\]
Observe that $t(r) \in \{0, 1, \cdots, p-1\}$ and $x^r = x^{t(r)}$ for every $x \in \mathbb{F}_p$. Therefore this transformation does not change a monomial as a function even though it reduces its degree.

\par Let $G(\mathbf{x})$ be the polynomial obtained by summing up the reduced monomials of $F(\mathbf{x})$. Note that $G(\mathbf{x})$ is still the indicator function of a solution for the system $(3)$ and $\deg G \le \deg F \le 2(p-1)^2$. Moreover $\deg_{x_i} G \le p-1$ for all $i$. Suppose that the system $(3)$ has no solution with $(x_1, \cdots, x_p) \ne (0, \cdots, 0)$. Let $1 \le k \le p$, define $\mathbf{x}_k = (x_1, \cdots, x_{k-1}, x_{k+1}, \cdots, x_{3p-3})$. Let $0 \ne c \in \mathbb{F}_p$. Write \[G(\mathbf{x}) = (x_k - c) Q(\mathbf{x}) + R(\mathbf{x}_k).\]

\par Since the system $(3)$ has no solutions with $x_k = c$ it follows that $R(\mathbf{x}_k) = 0$ for all $\mathbf{x}_k \in \mathbb{F}_p^{3p-4}$. However $\deg_{x_i} R \le p-1$ for all $i \ne k$. Lemma \ref{Davenports-Principle} implies that $R \equiv 0 \in \mathbb{F}_p[\mathbf{x}_k]$ and so $x_k - c$ divides $G(\mathbf{x})$. Finally, notice that $\mathbb{F}_p[\mathbf{x}]$ is a UFD and all the factors $x_k - c$ are distinct. Therefore there is a polynomial $H(\mathbf{x}) \in \mathbb{F}_p[\mathbf{x}]$ such that
\[G(\mathbf{x}) = (x_1^{p-1} - 1) \cdots (x_p^{p-1} - 1) H(\mathbf{x}). \tag{$4$}\]

\par We now find a contradiction using Lemma \ref{tech-lemma}. Notice that there exists $\lambda$ such that $b_i + \lambda a_i \ne 0$ for all $p+1 \le i \le 2p-1$. Therefore through a change of basis it can be assumed that the vectors $v_{p+1}, \cdots, v_{2p-1}$ do not lie in the projective class $[1: 0]$. In particular, $b_{p+1} \cdots b_{2p-1} \ne 0$, and no $p$ vectors lie in the same projective class. Thus Lemma \ref{tech-lemma} provides us with indices $\{i_1, \cdots, i_{p-1}\} \subset \{1, \cdots, p, 2p, \cdots, 3p-3\}$ such that
\[\sum_{I \subset \{1, 2\}^{2p-2} \atop \#\{i : I(i) = 1\} = p-1} \prod_{i = 1}^{p-1} v_{i+p}(I(i)) \prod_{j=1}^{p-1} v_{i_j}(I(p+j-1)) \ne 0.\]
However, this is exactly the coefficient of $x_{p+1}^{p-1} \cdots x_{2p-1}^{p-1}x_{i_1}^{p-1} \cdots x_{i_{p-1}}^{p-1}$ in the expansion of $F(\mathbf{x})$. Notice that this monomial has no variable with degree greater than $p-1$ and so it is preserved under the reduction process. Moreover this monomial has degree $2(p-1)^2 = \deg F$ and so it cannot be canceled by any other monomial after the reduction. In particular, it also appears in the expansion of $G(\mathbf{x})$.

\par By combining this with $(4)$ notice that $H(\mathbf{x})$ must contain a monomial multiple of $x_{p+1}^{p-1} \cdots x_{2p-1}^{p-1}$. Therefore $\deg H \ge (p-1)^2$ and so $\deg G \ge p(p-1) + (p-1)^2 > 2(p-1)^2$ contradicting $\deg G \le \deg F$. The proposition follows. 
\end{proof}

\section{Zero-Sum Sequences}\label{zero-sums}

\par This section lists the combinatorial results used in contracting variables. The contractions will involve at most $p$ variables at a time although a certain number of initial variables is necessary to guarantee the existence of such contractions. The first result is due to Olson (c.f. Lemma 1.1 in \cite{olson1969combinatorial}):

\begin{lemma}\label{olson} Let $v_1, \cdots, v_{3p-2} \in \mathbb{F}_p^2$. Then there exist distict indices $i_1, \cdots, i_t$ with $1 \le t \le p$ such that\[v_{i_1} + \cdots + v_{i_t} = \mathbf{0}.\] 
\end{lemma}

This result will be used to contract primary variables among themselves. Moreover this result has other consequences that will be used in contractions of secondary variables. The first consequence is described in the following lemma:

\begin{lemma}\label{olson1} Let $c_1, c_2, \cdots, c_{3p-2}$ be $p$-adic integers not divisible by $p$. Then there exist distinct indices $i_1, \cdots, i_t$ with $1 \le t \le p$ such that $$\ord \text{}_p (c_{i_1} + \cdots + c_{i_t}) = 1.$$
\end{lemma}

\begin{proof} For each $i$ write $c_i \equiv a_i + pb_i \pmod {p^2}$ with $a_i \in \{1, 2, \cdots, p-1\}$ and $b_i \in \{0, 1, \cdots, p-1\}$. Let $v_i = \binom{a_i}{b_i}$. By Lemma \ref{olson} there exist distinct indices $i_1, \cdots, i_t$ with $1 \le t \le p$ such that $v_{i_1} + \cdots + v_{i_t} \equiv \binom{0}{0} \pmod p$. Hence the number
\[c_{i_1} + \cdots + c_{i_t} \equiv a_{i_1} + \cdots + a_{i_t} + p(b_{i_1} + \cdots + b_{i_t}) \equiv a_{i_1} + \cdots + a_{i_t} \pmod {p^2}\] is a multiple of $p$. However $1 \le a_{i_1} + \cdots + a_{i_t} \le  p(p-1)$ and so $\text{ord}_p(c_{i_1} + \cdots + c_{i_t}) = 1$. 
\end{proof}

\begin{cor}\label{olson2} Let $v_1, \cdots, v_{3p-2} \in \mathbb{Z}_p^2$ be primitive vectors in the same projective class. Then there exist distinct indices $i_1, \cdots, i_t$ with $1 \le t \le p$ such that both entries of $v_{i_1} + \cdots + v_{i_t}$ are multiples of $p$ but at least one is not a multiple of $p^2$.
\end{cor}

\begin{cor}\label{olson3} Let $v_1, \cdots, v_{3p^2-2} \in \mathbb{Z}_p^2$ be primitive vectors. Then there exist distinct indices $i_1, \cdots, i_t$ with $1\le t \le p$ such that both entries of $v_{i_1} + \cdots + v_{i_t}$ are multiples of $p$ but at least one is not a multiple of $p^2$.
\end{cor}

\begin{proof} Corollary \ref{olson2} is immediate. Corollary \ref{olson3} follows from Corollary \ref{olson2} and the observation that there will be $\lceil \frac{3p^2-2}{p+1} \rceil = 3p-2$ vectors in the same projective class. 
\end{proof}

Corollary \ref{olson3} provides a quadratic bound on the minimum number of variables needed to make a contraction from level $l$ to level $l+1$. However, for the last two parts of Theorem \ref{main}, it is convenient to have a linear bound instead, and that will be derived from the following combinatorial result (c.f. Theorem 1 in \cite{alon1995lattice} and Equation 1.4 in \cite{chintamani2012new}):

\begin{lemma} \label{alon} There exists a constant $C_0$ such that for all $p$, every sequence of at least $C_0p$ elements of $\mathbb{F}_p^3$ contains a zero-sum subsequence of length $p$.
\end{lemma} 
Define $C = C_0 + 2$. Although it is conjectured that $C_0 \le 9$, the estimates in \cite{chintamani2012new} only show $C_0 \le 9994$, i.e. $C \le 9996$. The following result is a simple consequence of the previous lemma and it plays an important role in Section \ref{secondary-tau}.

\begin{lemma}\label{alon1} Let $v_1,\cdots, v_{Cp} \in \mathbb{Z}_p^2$ be primitive vectors. Then there exist distinct indices $i_1, \cdots, i_t$ with $1 \le t \le p$ such that the entries of $v_{i_1} + \cdots + v_{i_t}$ are both multiples of $p$ but at least one is not a multiple of $p^2$.
\end{lemma}

\begin{proof} We prove this result for $t = p$. For each $i$ write $v_i(1) \equiv a_i + pb_i \pmod{p^2}, v_i(2) \equiv c_i \pmod{p}$ with $a_i, b_i, c_i \in \{0, 1, \cdots, p-1\}$. Let \[w_i = \begin{pmatrix} a_i \\ b_i \\ c_i \end{pmatrix} \in \mathbb{F}_p^3, i = 1, \cdots, Cp.\] Lemma \ref{alon} implies that there exists $w_{i_1}, \cdots, w_{i_p}$ for which the entries of $w_{i_1} + \cdots + w_{i_p}$ are multiples of $p$. The first entry mod $p^2$ is given by $a_{i_1} + \cdots + a_{i_p}$ which lies in the interval $[0, p(p-1)]$. Hence if at least one of the entries $a_{i_j}$ is non-zero this sequence satisfies the desired properties. Otherwise the reductions of $v_{i_1}, \cdots, v_{i_p}$ are in the projective class $[0:1]$. In this case remove these vectors from the original set and repeat the argument above twice more. This process either provides a desired sequence or a set of $3p$ vectors in the same projective class. In the latter case Corollary \ref{olson2} provides a desired sequence, thus concluding the proof. 
\end{proof}

\par These results will be used to contract secondary variables among themselves. Different from the case of primary variables, generating secondary variables at levels higher than $\tau$ is not advantageous as such new variables will have no use. For this reason it is important to guarantee that contractions of secondary variables will provide a new variable at a bounded level. In fact, the property that $p^2$ does not divide both entries of certain sums of subsequences will be needed to show that the corresponding contraction generates a secondary variable exactly at the next level.

\section{Outline of The Proof of Theorem \ref{main}}\label{outline}

\par The main goal is to create a primary variable at a level higher than $\tau$. In order to do that it will be necessary to initially create $p^{\tau}$ primary variables at levels greater than zero. Then by contracting secondary variables from levels zero and one we generate secondary variables at levels $1, \cdots, \tau - 1$ which will be used to guarantee that contractions of primary variables can be made effectively. Then we contract primary variables to generate $p$ primary variables at level $\tau$. Lastly we contract the remaining variables to generate $2p-3$ secondary variables at level $\tau$. Theorem \ref{main} then follows by Proposition \ref{main-lemma}.

\par A few remarks concerning the different parts of Theorem \ref{main} are required at this point. For $\tau \ge 2$, the first three steps described in the previous paragraph will be the same. These three steps provide us with $p$ primary variables at level $\tau$ and uses $2p^{\tau+1} + 7p^{\tau}$ variables at levels zero and one, leaving a certain number of variables to be used in the last step. For $\tau = 1$, it will be sufficient to generate $p$ primary variables at level one and use the remaining variables to generate $2p-3$ secondary variables at level one. For each case of Theorem \ref{main} there is a subsection of Section \ref{secondary-tau} describing how to create $2p-3$ secondary variables at level $\tau$.

\par In summary, Section \ref{generating-primary} will be universal for all three parts of Theorem \ref{main}. Sections \ref{stepstones} and \ref{primary-tau} will be needed for $\tau \ge 2$ only. Moreover, since $C \ge 3,$ these sections only rely on the bound $s > 2\frac{p}{p-1}d^2 - 2d$. Lastly, Section \ref{secondary-tau} contains three different subsections dealing with the different cases of Theorem \ref{main}.

\section{Generating Primary Variables}\label{generating-primary}

\par Recall from Section \ref{variables} that $m_0 \ge \frac{s}{d}$ and $q_0 \ge \frac{s}{2d}$. By the assumptions of Theorem \ref{main}, we have $s > 2\frac{p}{p-1}d^2 - 2d$. Therefore \[m_0 \ge 2p^{\tau + 1} - 1 \text{ and } q_0 \ge p^{\tau + 1}.\] We now select a set $\mathcal{H}$ of at most $2p^{\tau + 1}$ variables at level zero and use it to create $p^{\tau}$ primary variables at levels greater than zero. The set $\mathcal{H}$ will satisfy $\# \mathcal{H} \ge 2p^{\tau+1} -1$ and $q(\mathcal{H}) \ge p^{\tau+1}$. The need to find such a set comes from the following combinatorial result of Davenport and Lewis (c.f. Lemma 5.1 in \cite{brudern2002artin}):

\begin{lemma}\label{primary-var} Let $\mathcal{H}$ be a set of variables at level zero. Then $\mathcal{H}$ contains at least \[\min \left(\left\lfloor\frac{\#\mathcal{H}}{2p-1}\right\rfloor, \left \lfloor \frac{q(\mathcal{H})}{p}\right \rfloor\right)\] pairwise disjoint contractions to primary variables at levels greater than zero.
\end{lemma}

\par As in \cite{brudern2002artin} let $I_0$ be the number of variables in the class $[1:0]$ at level zero. Without loss of generality it can be assumed that $[1:0]$ is the class containing the most variables at level zero so that $m_0 = I_0 + q_0$. In order to prove the existence of the set $\mathcal{H}$ we consider two cases.

\par First assume that $I_0 \ge p^{\tau + 1}$. Start removing variables from the class $[1:0]$ until there are only $p^{\tau + 1}$ variables left. Then remove variables from the other classes until there are only $p^{\tau + 1}$ variables left. Let $\mathcal{H}$ be the set of remaining variables. It is clear that $\#\mathcal{H} = 2p^{\tau + 1}$ and $q(\mathcal{H}) = p^{\tau + 1}$.

\par Second assume that $I_0 \le p^{\tau + 1} - 1$. Notice that $q_0 = m_0 - I_0 \ge p^{\tau+1}$ so that we can remove variables from classes different from $[1:0]$ until there are only $2p^{\tau+1} - 1$ variables left in total. Let $\mathcal{H}$ be this set of variables. Notice that $[1:0]$ is still the class with most variables in the remaining set. Therefore $\# \mathcal{H} = 2p^{\tau +1} - 1$ and $q(\mathcal{H}) \ge p^{\tau + 1}$.

\par The existence of such a set $\mathcal{H}$ guarantees that there is a set of at most $2p^{\tau + 1}$ variables at level zero which can be turned into a set of $p^{\tau}$ primary variables at higher levels. Let $t_l$ be the number of primary variables generated at level $l$ through this process. Observe that if $t_l \ge 1$ for $l \ge \tau+1$ there is nothing more to prove as this will be a primary variable at level at least $\tau + 1$. Hence it can be assumed that \[t_1 + \cdots + t_{\tau} = p^{\tau}.\]
\par The goal is to contract these primary variables until level $\tau + 1$ is reached. The main method for contracting these variables is given by Lemma \ref{olson} and consists of selecting a group of $3p-2$ primary variables and finding a subsequence of at most $p$ of them summing up to zero. However, once the number of primary variables is less than $3p-2$ but at least $p$, Proposition \ref{main-lemma} makes such contractions still possible as long as there are a few secondary variables available at that level.

\par The contraction of primary variables will be described in Section \ref{primary-tau}, after the proper arrangements with secondary variables have been taken care of in Section \ref{stepstones}.

\section{Secondary Variables as Stepstones}\label{stepstones}

\par The goal of this section is to create secondary variables at levels $1, 2, \cdots, \tau - 1$ for $\tau \ge 2$. These will be used as stepstones for the contractions of primary variables. More precisely, we will create at least $3p^2-p-2$ secondary variables at levels $1, 2, \cdots, \tau - 2$ and at least $4p-6$ secondary variables at level $\tau - 1$.

\par As previously remarked, secondary variables at levels greater than $\tau$ have no use in the argument. For this reason, it is important to make sure that a contraction of secondary variables is at a bounded level. The following result follows immediately from Corollary \ref{olson3}.

\begin{lemma}\label{contraction1}
Consider $m \ge 3p^2 - 2$ variables at the same level $l$. These variables can be contracted into \[\left \lceil \frac{m+3 - 3p^2}{p} \right \rceil\] variables at level $l+1$. Moreover after these contractions at least $3p^2 - p - 2$ variables are left at level $l$.
\end{lemma}

\par Recall from Section \ref{variables} that $m_0 + m_1 \ge 2\frac{s}{d}$.  In particular, \[(m_0 - 2p^{\tau+1}) + m_1 \ge 2p^{\tau+1}-3 \ge 7p^{\tau}.\]

\par Select $7p^{\tau}$ variables from levels zero and one which are not in $\mathcal{H}$. Notice that by Lemma \ref{contraction1} these can be contracted into $7p^{\tau - 1} - 3p$ variables at level $1$. These variables can be contracted into $7p^{\tau-2} - 3p - 3$ variables at level $2$, and in turn these can be contracted into $7p^{\tau - 3} - 3p - 3$ variables at level $3$. Inductively, the levels $1, 2, \cdots, \tau - 2$ will have at least $3p^2 - p -2$ variables left from these contractions and level $\tau - 1$ will have at least $7p - 3p - 3 > 4p-6$ variables. 

\par These variables will now be used for contractions of primary variables. This is described in detail in the next section.

\section{Primary Variables at Level $\tau$}\label{primary-tau}

\par The goal of this section is to generate $p$ primary variables at level $\tau$ for $\tau \ge 2$. First, notice that as long as there are $3p-2$ primary variables at level $l$, Lemma \ref{olson} guarantees a contraction using at most $p$ primary variables. Then at the moment these contractions can no longer be made, there will be at least $2p-2$ and at most $3p-3$ primary variables at that level. In that case at most two contractions will be made using the secondary variables constructed in Section \ref{stepstones}. This is described below.

\par In case $1 \le l \le \tau - 2$, Section \ref{stepstones} ensures that there are at least $3p^2-p-2$ secondary variables at level $l$. As $3p^2-p-2 > p^2-p-1$ there are $p-1$ secondary variables in the same projective class. By Corollary \ref{C-D-3} there is a contraction using at most $p$ primary variables and $p-1$ secondary variables. After that, in case there are still $p$ primary variables remaining at level $l$, there is one more contraction following the same argument. This is because there are still $p-1$ secondary variables in the same projective class.

\par In case $l = \tau - 1$, Section \ref{stepstones} ensures that there are at least $4p-6$ secondary variables at level $l$. Since at most two more contractions will be made, $2p-3$ secondary variables will be assigned to each possible group of $p$ primary variables. Using Proposition \ref{main-lemma} one or two contractions will be made at this level, according to the number of primary variables available.

\par With such contractions in mind notice that if at any point a level $l \le \tau - 1$ contains $t$ primary variables, then these variables can be contracted into $\lfloor \frac{t}{p} \rfloor$ primary variables at higher levels. Without loss of generality assume that no contraction will generate a primary variable at a level $l \ge \tau +1$ as that variable guarantees a solution for the system $(2)$ immediately. We now present the following result:

\begin{lemma}\label{contraction2}
Consider $p^{\tau}$ variables distributed over levels $1, 2, \cdots, \tau$. Assume that a level $l < \tau$ has $t$ variables at some point throughout the contractions. Then these variables can be contracted into $\lfloor \frac{t}{p} \rfloor$ variables at higher levels. This can be done at least once per level. Then there exists a sequence of contractions which generates $p$ variables at level $\tau$.
\end{lemma}

\begin{proof}
The proof goes by induction on $\tau$. For $\tau = 1$ the result is immediate. Assume that it holds for $\tau - 1$ and let $t_l$ be the number of variables at level $l$. Define $t_1 = x, t_2 + \cdots + t_l = p^{\tau} - x$. A contraction at the first level gives us \[p^{\tau} - x + \left \lfloor \frac{x}{p} \right \rfloor \ge p^{\tau - 1}\] variables at levels $2, 3, \cdots, \tau$. The result follows by ignoring the first level and using the induction hypothesis.
\end{proof}

\par Recall that in Section \ref{generating-primary} we constructed $p^{\tau}$ primary variables at levels $1, \cdots, \tau$. This combined with Lemma \ref{contraction2} proves that we can generate $p$ primary variables at level $\tau$. The last step of the argument consists of creating $2p-3$ secondary variables at level $\tau$. To that end, different methods for contractions will be employed for the different cases of Theorem \ref{main}. Section \ref{secondary-tau} covers each of these cases separately.

\section{Secondary Variables at Level $\tau$} \label{secondary-tau}

\par The goal of this section is to generate $2p-3$ secondary variables at level $\tau$ for each of the cases described in Theorem \ref{main}. This, combined with the $p$ primary variables already at that level and Proposition \ref{main-lemma}, will be enough to finish the proof of Theorem \ref{main}. 

\par Notice that for $\tau \ge 2$ we have only used $2p^{\tau+1} + 7p^{\tau}$ variables so far from levels zero and one, and for $\tau = 1$ we have only used $2p^2$ variables from level zero to do contractions. The normalization of the system $(1)$ gives us $m_0 + \cdots + m_{\tau} \ge (\tau + 1) \frac{s}{d}$. We now study each case of Theorem \ref{main} separately in the following subsections.

\subsection{} Let $\tau \ge 3, s > 2\frac{p}{p-1}d^2 - 2d$, and $p \ge 7$.

\par Observe that  \[(m_0 + m_1 - 2p^{\tau+1} - 7p^{\tau}) + m_2 + \cdots + m_{\tau} \ge 2\tau p^{\tau+1} - 7p^{\tau} - 2\tau -1 > 6p^{\tau + 1} - 8p^{\tau}.\] 

\par Consider now a set of $6p^{\tau+1} - 8p^{\tau}$ secondary variables available. Let $s_l$ be the number of these variables at level $l$. We now use Lemma \ref{contraction1} to contract variables at levels $0, 1, \cdots, \tau - 1$ in this order. It is clear that this process generates at least

\[\frac{\frac{\frac{s_0+3-3p^2}{p}+s_1+3-3p^2}{p}+ \cdots + s_{\tau-1} + 3 - 3p^2}{p} + s_{\tau}\]
secondary variables at level $\tau$. Notice that $s_0 + ps_1 + \cdots + p^{\tau}s_{\tau} \ge s_0+s_1 + \cdots + s_{\tau} = 6p^{\tau+1} - 8p^{\tau}$ and moreover most of the terms $\frac{3}{p^k}$ cancel out in the expression above. Hence there will be at least $6p-8 - 3p-3+ \frac{3}{p^{\tau}}+\frac{3}{p^{\tau-1}}  > 3p-11$ secondary variables at level $\tau$. Since $3p - 10 \ge 2p-3$ there will be $2p-3$ secondary variables at that level. The result follows.

\subsection{} Let $\tau = 2, s > 2\frac{p}{p-1} d^2 - 2d$, and $p \ge \frac{C}{2} + 4$.

\par As in the previous case there will be at least \[(m_0+m_1 - 2p^3 - 7p^2) + m_2 \ge 4p^3 -7p^2-5 > 4p^{3}-8p^2\] secondary variables available. The following result is an immediate consequence of Lemma \ref{alon1} and allows us to do contractions effectively for large values of $p$.

\begin{lemma} \label{contraction3}
Consider $m$ variables at the same level $l$. These variables can be contracted into \[\left \lceil \frac{m-Cp}{p} \right \rceil\] variables at level $l+1$.
\end{lemma}

Keeping the notation of the previous subsection, this provides us with at least
\[\frac{\frac{s_0-Cp}{p}+s_1-Cp}{p}+ s_{2} \ge 4p-8 - C - \frac{C}{p}\]
secondary variables at level $\tau$. Since $4p-8-C-\frac{C}{p} \ge 2p - 3$ the result follows as before.

\subsection{} Let $\tau = 1, s > (2\frac{p}{p-1} + \frac{C-3}{2p-2})d^2 - 2d$ and $p \ge 5$.

\par As previously remarked, this case only relies on the argument presented in Section \ref{generating-primary}, i.e. we have created $p$ primary variables at level $\tau = 1$ using $2p^2$ primary variables at level zero. There are \[m_0 + m_1 - 2p^2  \ge 2p^2 + (C-3)p -3\] variables available at levels zero and one. By Lemma \ref{contraction3} we can generate $2p-3$ secondary variables at level $\tau$. This completes the proof of Theorem \ref{main}.

\section*{Acknowledgment}

The author happily expresses his gratitude to Professor Christopher Skinner for introducing him to the topic and for helping with many inspiring conversations.

\bibliographystyle{unsrt}
\bibliography{artins_conjecture}{}

\end{document}